\documentclass[draft]{article}

\usepackage{amsfonts}
\usepackage{amssymb}
\usepackage{amsmath}
\usepackage{amsthm} 

\usepackage{graphicx}

\newtheoremstyle{mystyle}
  {}{}{\itshape}{}{\bfseries}{.}{ }{\thmnumber{#2}}
\theoremstyle{mystyle}

\newtheorem{theorem}{Theorem}[section]
\newtheorem{corollary}{Corollary}[theorem]
\newtheorem{lemma}[theorem]{Lemma}

\usepackage{authblk}    

\begin{document}

\title{Properties of the cumulated\\ deficient binary digit sum}
\date{march 2019}
\author[*]{Thomas Baruchel}
\affil[*]{\small \'Education nationale, France\authorcr\texttt{baruchel@riseup.net}}
\maketitle

\begin{abstract}
    The sequence \texttt{A268289} from the \textit{On-Line Encyclopedia of Integer Sequences}, namely the cumulated differences between the number of digits~$1$ and the number of digits~$0$ in the binary expansion of consecutive integers, is studied here. This sequence happens to match a sequence of cardinalities of some specific sets. Furthermore, it can also be expressed by using the Takagi function. The three different definitions have their own properties and combining them together lead to some new identities.
\end{abstract}

\section{Introduction}
%
The sequence \texttt{A268289} in the \textit{On-Line Encyclopedia of Integer Sequences} \cite{oeis} may be referred to as the cumulated deficient binary digit sum; each term of index~$n$ is the difference between the total number of digits~$1$ and the total number of digits~$0$ when writing down integers starting from~$1$ up to~$n$. Thus, $\texttt{A268289}_5=3$ because the digit~$1$ occurs $7$ times, while the digit~$0$ occurs $4$ times, in the following expansions: $1_2, 10_2, 11_2, 100_2, 101_2$.

These terms are known to be related to the $\tau$ Takagi function, introduced by T.~Takagi in 1901~\cite{takagi}. An extensive survey of its properties has been published recently by Jeffrey~C. Lagarias~\cite{lagarias}, and the section~9 of this work gives some formulas (from 1949 and 1968) using this~$\tau$ function for computing \emph{binary digit sums} directly related to the sequence \texttt{A268289}.

Furthermore, in a previously published paper studying Karatsuba's algorithm --- see formula (5) at the end of the section 4 in \cite{baruchel} --- was defined the following sequence of sets gathering some specific nodes in a recursion tree:
\begin{equation}\label{def_sets}
    S_n = \left\{  m \,\Big| \, 1\leqslant m\leqslant n,  \left(\left(n-m\right) \,\textrm{mod}\,\, 2^{\lfloor\log_2 m\rfloor+1}\right) < 2^{\lfloor\log_2 m\rfloor}\right\}
    \,\textrm{.}
\end{equation}
The sequence of cardinalities $\vert S_d\vert$ was then empirically found to be the sequence \texttt{A268289}. It was beyond the purpose of this paper to investigate further such an identity, which was left there as a mere conjecture. The identity~\ref{identity} below will prove that both sequences are actually identical.

The purpose of the current paper is to gather several new identities ---~mostly related to the cumulated deficient binary digit sum~--- coming from combining the properties of \texttt{A268289}, the $\tau$~function and the new sequence of sets~$S_n$.

\section{Preliminary identities}
\begin{lemma}\label{lemma2}
    Let $n$ be some positive integer and $\tau$ the Takagi function; then
\[
    {\normalfont\texttt{A268289}}_n =
       \left(n+1\right)\left(m - k + 1\right) - \left(2 + \tau(\xi)\right) 2^m  + 2^{k+1} - 1
\]
with $k=\lfloor\log_2(n)\rfloor$, $m=\lfloor\log_2(n+1)\rfloor$, $\xi=(n+1)2^{-m}-1$.
\end{lemma}
\begin{proof}
The sequence \texttt{A268289} is known to be closely related to the Takagi function. While not currently present in the online description of the sequence \texttt{A268289}, this formula is easy to build with the help of the theorem~9.1 in \cite{lagarias} as actually stated in a comment of the sequence.
\end{proof}

\begin{lemma}\label{lemma1}
    Let $k$ and $n$ be some positive integers such than $n<2^k$; then
\[
    {\normalfont\texttt{A268289}}_{n+2^k}
    = {\normalfont\texttt{A268289}}_n
        + \left(n+1\right) \left(\lfloor\log_2(n)\rfloor-k+2\right) + 2^k - 2^{\lfloor\log_2(n)\rfloor+1}
        \,\textrm{.}
\]
\end{lemma}
\begin{proof}
    This is proved by enumerative means: we find that $\texttt{A268289}_{2^k-1}=2^k-1$ by studying separately each length of binary expansions; then we add a single term and find that $\texttt{A268289}_{2^k}=2^k-k$, we then focus on the $n$ following terms by noticing that the binary expansions of integers $2^k+1, 2^k+2,\dots,2^k+n$ are strongly related to those of $1, 2, \dots, n$ which are involved in $\texttt{A268289}_n$, and we thus only have to take care of their leading digits $\texttt{1000}\dots$
\end{proof}

\begin{lemma}\label{lemma3}
    Let $k$ and $n$ be some positive integers such than $n<2^k$; then
\[
        \vert S_{n+2^k}\vert = \vert S_n\vert
        + \left(n+1\right) \left(\lfloor\log_2(n)\rfloor-k+2\right) + 2^k - 2^{\lfloor\log_2(n)\rfloor+1}\,\textrm{.}
\]
\end{lemma}
\begin{proof}
According to the definition~(\ref{def_sets}), and by noticing that an integer $e$ such that $e\leqslant n$ belongs to $S_{n+2^k}$ if and only if it also belongs to $S_n$, we find
\[
    \begin{array}{lll}
        S_{n+2^k}= S_n
        &  \cup &
        \bigcup_{j=\lfloor\log_2(n)\rfloor+1}^{k-1} \left\{ 2^j+n+1, 2^j+n+2,\dots,2^{j+1}-1  \right\} \\[10pt]
        & \cup &\left\{2^k, 2^k+1,\dots,2^k+n\right\}
\end{array}
\]
and summing the cardinalities of all these subsets yields the expected result.
\end{proof}

\begin{theorem}\label{identity}
For every nonnegative integer $n$, ${\normalfont\texttt{A268289}}_n = \vert S_n\vert$.
\end{theorem}
\begin{proof}
    The identities~\ref{lemma1} and~\ref{lemma3} both give a rule for building any term of the corresponding sequences by starting from the three initial terms. The three initial terms $(0,1,1)$ are the same, and the building rule is identical in both sequences; thus both sequences are identical.
\end{proof}

\begin{theorem}\label{digit0}
    Let $n$ be a positive integer such that $n<3\times 2^{\lfloor\log_2(n)\rfloor-1}$, then
    \[
        {\normalfont\texttt{A268289}}_{n+2^{\lfloor\log_2(n)\rfloor-1}}
        =
        {\normalfont\texttt{A268289}}_n
    +2\left(n+1\right)-2^{\lfloor\log_2(n)\rfloor+1}
    \,\textrm{.}
    \]
\end{theorem}
\begin{proof}
    According to the identity~\ref{identity}, $\texttt{A268289}_n=\vert S_n\vert$, and thus we refer to the definition of~$S_n$ in~(\ref{def_sets}). Adding the considered power of~$2$ to~$n$ keeps all congruence relations unchanged except for the largest integers, namely for~$m\geqslant 2^{\lfloor\log_2(n)\rfloor-1}$. Thus, only two subsets of~$S_n$ have to be studied:
    \begin{itemize}
        \item integers~$m$ such that $2^{\lfloor\log_2(n)\rfloor-1}\leqslant m<2^{\lfloor\log_2(n)\rfloor}$ belong to~$S_{n+2^{\lfloor\log_2(n)\rfloor-1}}$ if and only if they do not belong to~$S_n$, thus the new set is built by removing~$3\times 2^{\lfloor\log_2(n)\rfloor-1}-n-1$ elements and by adding $-2^{\lfloor\log_2(n)\rfloor}+n+1$ new elements;
        \item integers~$m$ such that $2^{\lfloor\log_2(n)\rfloor}\leqslant m\leqslant n$ belong simultaneously to~$S_n$ and $S_{n+2^{\lfloor\log_2(n)\rfloor-1}}$ while $2^{\lfloor\log_2(n)\rfloor-1}$ integers greater than~$n$ are also added to the set $S_{n+2^{\lfloor\log_2(n)\rfloor-1}}$.
    \end{itemize}
    Thus $\vert S_{n+2^{\lfloor\log_2(n)\rfloor-1}}\vert - \vert S_n\vert = 2\left(n+1\right)-2^{\lfloor\log_2(n)\rfloor+1}$ as expected.
\end{proof}

\section{Useful identities involving the Takagi function}\label{takagisection}

\begin{theorem}\label{main}
For every real number $\xi\in[0,1]$ and any positive integer $m$,
\[
    \tau((\xi+1)2^{-m}) = 2^{-m}\left(m\left(\xi+1\right)-2\xi+\tau(\xi)\right)
    \,\textrm{.}
\]
\end{theorem}
\begin{proof}
    By combining the identities~\ref{lemma1} and~\ref{lemma2}, and after some simplification, we find that for every dyadic rational $\xi=k/2^m$ in $\left[0,1\right]$,
    \[
        \tau(\xi)
        = \xi\left(m-\lfloor\log_2(k)\rfloor-2\right) + 2^{\lfloor\log_2(k)\rfloor-m}\left(2+\tau\left(k\, 2^{-\lfloor\log_2(k)\rfloor}-1\right)\right)
    \]
from which comes that for every dyadic rational number $\xi\in[1,2]$ and any positive integer $m$,
    \[
        \tau(2^{-m}\xi) = 2^{-m}\left(m\xi-2\xi+2+\tau(\xi-1)\right)
        \,\textrm{.}
    \]
The Takagi function being continuous, it is then easy to enclose any real from the same interval between two dyadic rational numbers at an arbitrary precision.
\end{proof}

\begin{theorem}\label{half}
    For every real number $\xi\in[0,1/2]$,
\[
    \tau(\xi+1/2) = 1/2 -2\xi+\tau(\xi) \,\textrm{.}
\]
\end{theorem}
\begin{proof}
    By combining the identities~\ref{lemma2} and~\ref{digit0}, we find, after some simplification, that for~$n$ such that $n<3\times 2^{\lfloor\log_2(n)\rfloor-1}$,
    \[
        \tau\left( \displaystyle\frac{n+1}{2^{\lfloor\log_2(n)\rfloor}}-\displaystyle\frac{1}{2}\right)
        =
        \displaystyle\frac{5}{2} - \frac{2(n+1)}{2^{\lfloor\log_2(n)\rfloor}}
        + \tau\left( \displaystyle\frac{n+1}{2^{\lfloor\log_2(n)\rfloor}}-1\right)
    \]
and since the Takagi function is continuous, we prove the expected identity.
\end{proof}

\begin{theorem}\label{general}
    For every positive integer~$n$,
    \[
        {\normalfont\texttt{A268289}}_n
        =
        n - 2^k \tau\left(\displaystyle\frac{n+1} {2^k }   - 1  \right)
    \]
    with $k=\lfloor\log_2(n)\rfloor$.
\end{theorem}
\begin{proof}
    We work on two separate cases; let initially $n$ be a positive integer such that $n<3\times 2^{\lfloor\log_2(n)\rfloor-1}$, then, by combining the two identities~\ref{lemma2} and~\ref{digit0}, we get after some simplification:
    \[
        {\normalfont\texttt{A268289}}_n
        =
        -n - 2 + 2^k \left( \displaystyle\frac{5}{2}
        -\tau\left(\displaystyle\frac{n+1} {2^k }   -\frac{1}{2}  \right) \right)
    \]
    with $k=\lfloor\log_2(n)\rfloor$. Then we clean the resulting expression with the help of the theorem~\ref{half} and find the expected expression.

    Let now $n$ be a positive integer such that $n\geqslant 3\times 2^{\lfloor\log_2(n)\rfloor-1}$, then, by combining again the two identities~\ref{lemma2} and~\ref{digit0}, we get after some simplification:
    \[
        {\normalfont\texttt{A268289}}_n
        =
        3n + 2 - 2^k \left( \displaystyle\frac{7}{2}
        +\tau\left(\displaystyle\frac{n+1} {2^k }   -\frac{3}{2}  \right) \right)
    \]
    with $k=\lfloor\log_2(n)\rfloor$. We clean again the resulting expression with the help of the theorem~\ref{half} and find the expected expression.
\end{proof}
\begin{corollary}\label{majorlink2}
    For every dyadic rational $\xi=n/2^k$ in $\left[0,1\right]$,
    \[
        \tau(\xi)=
        1 + \xi - \displaystyle\frac{1+{\normalfont\texttt{A268289}}_{n+2^k-1}}
        {2^k} \,\textrm{.}
    \]
\end{corollary}

\begin{theorem}\label{initial_identity}
For every nonnegative integer~$n$,
\[
    {\normalfont\texttt{A268289}}_{n+2^{m+1}} = 2n- 2^m\tau(\xi) + 1
\]
with $m=\lfloor\log_2(n+1)\rfloor$, $\xi=(n+1)2^{-m} - 1$ and $\tau$ being the Takagi function.
\end{theorem}
\begin{proof}
    This comes from combining again the two identities~\ref{lemma1} and~\ref{lemma2}.
\end{proof}
\begin{corollary}\label{majorlink}
    For every dyadic rational $\xi=n/2^k$ in $\left[0,1\right]$,
    \[
        \tau(\xi) = 2+2\xi - \displaystyle\frac{ 1 + {\normalfont\texttt{A268289}}_{2^{k+1}+2^k+n-1} }{ 2^k } \,\textrm{.}
    \]
\end{corollary}

\section{Main identities related to \texttt{A268289}}

\begin{theorem}\label{oeis1}
    Let $n$ and $k$ be some nonnegative integers such that $n\leqslant 2^k$,
    \[
        {\normalfont\texttt{A268289}}_{2^{k+2}-n-1}
        = 2^{k+1} - 4n + {\normalfont\texttt{A268289}}_{2^{k+1}+2^k+n-1}
        \,\textrm{.}
    \]
\end{theorem}
\begin{proof}
    This merely comes from applying the fundamental functional identity $\tau(\xi)=\tau(1-\xi)$ (see theorem~4.1 in \cite{lagarias}) to the identity~\ref{majorlink}.
\end{proof}

\begin{theorem}\label{oeis2}
    Let $n$ and $k$ be some nonnegative integers such that $n\leqslant 2^k$,
    \[
        {\normalfont\texttt{A268289}}_{2^{k+2}+2^{k+1}+n-1} 
        =
        2^{k+1} - n +  {\normalfont\texttt{A268289}}_{2^{k+1}+2^k+n-1} 
      \,\textrm{.}
    \]
\end{theorem}
\begin{proof}
    This merely comes from applying the fundamental functional identity $\tau(\xi/2)=\xi/2+\tau(\xi)/2$ (see theorem~4.1 in \cite{lagarias}) to the identity~\ref{majorlink}.
\end{proof}

\begin{theorem}\label{oeis4}
    Let $n$ and $k$ be some nonnegative integers such that $n\leqslant 2^k$,
    \[
        {\normalfont\texttt{A268289}}_{2^{k+1}-n-1}
        =2^k-2n+
        {\normalfont\texttt{A268289}}_{2^k+n-1}
        \,\textrm{.}
    \]
\end{theorem}
\begin{proof}
    This merely comes from applying the fundamental functional identity $\tau(\xi)=\tau(1-\xi)$ (see theorem~4.1 in \cite{lagarias}) to the identity~\ref{majorlink2}.
\end{proof}

\begin{theorem}\label{oeis5}
    Let $n$ and $k$ be some nonnegative integers such that $n\leqslant 2^k$,
    \[
        {\normalfont\texttt{A268289}}_{2^{k+1}+n-1}
        =2^k - n+
        {\normalfont\texttt{A268289}}_{2^k+n-1}
        \,\textrm{.}
    \]
\end{theorem}
\begin{proof}
    This merely comes from applying the fundamental functional identity $\tau(\xi/2)=\xi/2+\tau(\xi)/2$ (see theorem~4.1 in \cite{lagarias}) to the identity~\ref{majorlink2}.
\end{proof}

\begin{theorem}\label{oeis6}
    Let $n$ and $k$ be some nonnegative integers such that $n\leqslant 2^k$,
    \[
        {\normalfont\texttt{A268289}}_{2^{k+1}+n-1}
        = n+
        {\normalfont\texttt{A268289}}_{2^{k+1}-n-1}
        \,\textrm{.}
    \]
\end{theorem}
\begin{proof}
    This comes from combining theorems~\ref{oeis4} and~\ref{oeis5}.
\end{proof}

\begin{theorem}\label{oeis7}
    Let $n$ and $k$ be some nonnegative integers such that $n\leqslant 2^k$,
    \[
        {\normalfont\texttt{A268289}}_{2^{k+1}+2^k+n-1}
        = 2n 
        + {\normalfont\texttt{A268289}}_{2^{k+1}+n-1}
        \,\textrm{.}
    \]
\end{theorem}
\begin{proof}
    This comes from combining theorems~\ref{oeis1} and~\ref{oeis4}.
\end{proof}

\begin{theorem}\label{oeis3}
    Let $n$ and $k$ be some nonnegative integers such that $n\leqslant 2^k$,
    \[
 {\normalfont\texttt{A268289}}_{2^{k+1}+2^k+n-1} 
        =
        3n +  {\normalfont\texttt{A268289}}_{2^{k+1}-n-1}
     \,\textrm{.}
    \]
\end{theorem}
\begin{proof}
    This comes from combining theorems~\ref{oeis6} and~\ref{oeis7}.
\end{proof}

\begin{theorem}\label{oeis8}
    Let $n$ and $k$ be some nonnegative integers such that $n\leqslant 2^k$,
    \[
        {\normalfont\texttt{A268289}}_{2^{k+2} +2^k-n-1}
        =
        2^k+n
        + {\normalfont\texttt{A268289}}_{2^{k+1}+n-1}
        \,\textrm{.}
    \]
\end{theorem}
\begin{proof}
    This comes by choosing some $n'$ and $k'$ such that index~$2^{k'+1}-n'-1$ in identity~\ref{oeis6} matches index~$2^{k+1}+2^k+n-1$ in identity~\ref{oeis7}.
\end{proof}

\section{Regular iterations over specific indices}

\begin{theorem}\label{oeis9}
    Let $n$ and $k$ be some nonnegative integers such that $n\leqslant 2^k$,
    \[
        {\normalfont\texttt{A268289}}_{2^{k+3} +2^k+n-1}
        =
        2^{k+2}
        + {\normalfont\texttt{A268289}}_{2^{k+1}+n-1}
        \,\textrm{.}
    \]
\end{theorem}
\begin{proof}
    This comes by iterating twice with theorem~\ref{oeis8}, using $n\gets 2^k-n$ and $k\gets k+1$ for the second step.
\end{proof}

\begin{theorem}\label{oeis11}
    Let $n$ and $k$ be some nonnegative integers such that $n\leqslant 2^k$,
    \[
        {\normalfont\texttt{A268289}}_{2^{k+3} -2^k-n-1}
        =
        3\times 2^{k}
        + {\normalfont\texttt{A268289}}_{2^{k+1}-n-1}
        \,\textrm{.}
    \]
\end{theorem}
\begin{proof}
    This comes by iterating twice with theorem~\ref{oeis3}, using $n\gets 2^k-n$ and $k\gets k+1$ for the second step.
\end{proof}

\begin{theorem}\label{oeis10sum}
    Let $n$ and $k$ and $m$ be some nonnegative integers such that $n\leqslant 2^k$, then
    \[
        {\normalfont\texttt{A268289}}_{2^{k+2m+1} + 2^k(4^m-1)/3  +n-1}
        =
        2^{k+2}(4^m-1)/3
        + {\normalfont\texttt{A268289}}_{2^{k+1}+n-1}
        \,\textrm{.}
    \]
\end{theorem}
\begin{proof}
    This comes from iterating $m$ times with the identity~\ref{oeis9}.
\end{proof}

\begin{theorem}\label{oeis12sum}
    Let $n$ and $k$ and $m$ be some nonnegative integers such that $n\leqslant 2^k$, then
    \[
        {\normalfont\texttt{A268289}}_{2^{k+2m+1} - 2^k(4^m-1)/3  -n-1}
        =
        2^{k}(4^m-1)
        + {\normalfont\texttt{A268289}}_{2^{k+1}-n-1}
        \,\textrm{.}
    \]
\end{theorem}
\begin{proof}
    This comes from iterating $m$ times with the identity~\ref{oeis11}.
\end{proof}

\begin{theorem}\label{oeis14sum}
    Let $n$ and $k$ and $m$ be some nonnegative integers such that $n\leqslant 2^k$, then
    \[
        {\normalfont\texttt{A268289}}_{2^{k+m}  +n-1}
        =
        2^{k}(2^m-1) -mn
        + {\normalfont\texttt{A268289}}_{2^k+n-1}
        \,\textrm{.}
    \]
\end{theorem}
\begin{proof}
    This comes from iterating $m$ times with the identity~\ref{oeis5}.
\end{proof}

\begin{theorem}\label{oeis13sum}
    Let $n$ and $k$ and $m$ be some nonnegative integers such that $n\leqslant 2^k$, then
    \[
        {\normalfont\texttt{A268289}}_{2^{k+m+1} + 2^{k+m}  +n-1}
        =
        2^{k+1}(2^m-1) -mn
        + {\normalfont\texttt{A268289}}_{2^{k+1}+2^k+n-1}
        \,\textrm{.}
    \]
\end{theorem}
\begin{proof}
    This comes from iterating $m$ times with the identity~\ref{oeis2}.
\end{proof}



\begin{theorem}
For every real number $\xi\in[0,1]$ and any nonnegative integer $m$,
    \[
    1 - 2\,\tau\!\left(\frac{1}{6} + \frac{3\xi-1}{6\times4^m}\right)
    =\frac{ 1-2\,\tau\!\left( \xi/2 \right)  }{4^m}
    =\frac{ 1-\xi-\tau( \xi)  }{4^m}
    \,\textrm{.}
\]
\end{theorem}
\begin{proof}
    This comes from combining the identity~\ref{oeis10sum} with the formula~\ref{general} after having defined~$\xi=n/2^k$; the formula is then extended to real numbers since the~$\tau$ function is continuous.
\end{proof}

\begin{theorem}
For every real number $\xi\in[0,1]$ and any nonnegative integer $m$,
    \[
        \frac{2}{3} - \tau\!\left(\frac{2}{3} + \frac{1/3-\xi}{4^m}\right)
    =\frac{ 2/3-\tau\!\left( 1-\xi \right)  }{4^m}
    \,\textrm{.}
\]
\end{theorem}
\begin{proof}
    This comes from combining the identity~\ref{oeis12sum} with the formula~\ref{general} after having defined~$\xi=n/2^k$; the formula is then extended to real numbers since the~$\tau$ function is continuous.
\end{proof}

\begin{theorem}
    For any nonnegative integer~$m$, ${\normalfont\texttt{A268289}}_{5\times 4^m /3 - 2/3}=4^m$. (These indices are each second term in the sequence~{\normalfont\texttt{A081254}}; furthermore, discarding the two trailing decimal digits of these indices also gives the sequence~{\normalfont\texttt{A033114}}.)
\end{theorem}
\begin{proof}
    This is obvious with~$(n,k)=(1,2)$ and $m\gets m-1$ in identity~\ref{oeis12sum}
\end{proof}

\section{Additional properties}

\begin{theorem}\label{encadrement}
    For every nonnegative integer~$n$, $\,\,\, n/2\leqslant {\normalfont\texttt{A268289}}_n \leqslant n$.
\end{theorem}
\begin{proof}
    The right part of the relation is easy to prove: we know from the definition~(\ref{def_sets}) that $S_n$ contains at most $n$ elements, furthermore $\texttt{A268289}_n=\vert S_n\vert$ according to the identity~\ref{identity}.

    The left part is proved by induction: we assume that $n/2\leqslant\texttt{A268289}_n$ is true for all terms up to some $\texttt{A268289}_{2^{k+1}-1}$. Then, the relation also stands for all terms up to $\texttt{A268289}_{2^{k+1}+2^k-1}$ according to the theorem~\ref{oeis6}. From this, we go further up to~$\texttt{A268289}_{2^{k+2}-1}$ with the help of the theorem~\ref{oeis4}. The relation being true for the initial terms, it is then true for all terms.
\end{proof}

\begin{theorem}\label{major}
    For every dyadic rational $\xi=n/2^k$ in $\left[0,1\right]$,
    \[
        \tau(\xi)\leqslant\frac{\xi+1}{2} - \frac{1}{2^{k+1}}
        \,\textrm{.}
    \]
\end{theorem}
\begin{proof}
    This comes from combining identities~\ref{majorlink2} and~\ref{encadrement}.
\end{proof}

\begin{theorem}\label{minor}
    Let $n$ be a positive integer such that $n\geqslant 3\times 2^{\lfloor\log_2(n)\rfloor-1}$, then
    \[
        {\normalfont\texttt{A268289}}_{n} \geqslant
        1 + \frac{ 3\left(n-2^{\lfloor\log_2(n)\rfloor}\right)}{2}
        \,\textrm{.}
    \]
\end{theorem}
\begin{proof}
    By combining identities~\ref{initial_identity} and~\ref{major}, we find that for every nonnegative integer~$n$,
    \[
        {\normalfont\texttt{A268289}}_{n+2^{\lfloor\log_2(n+1)\rfloor+1}} \geqslant 1+\frac{3n}{2}
        \,\textrm{.}
    \]
    Since $\lfloor\log_2(n)\rfloor$ is a more convenient expression than $\lfloor\log_2(n+1)\rfloor$, we do the substitution after having checked that the relation above is still true when $n$ is some $2^k-1$: in such cases, we know by enumerative means that $\texttt{A268289}_n=n$ and the relation is then obviously true.
\end{proof}


\begin{lemma}\label{lemma_half}
    The identity ${\normalfont\texttt{A268289}}_n=n/2$ implies $n < 3\times 2^{\lfloor\log_2(n)\rfloor-1}-1$.
\end{lemma}
\begin{proof}
    This comes directly from the identity~\ref{minor}, but it has to be checked separately that $n=3\times 2^{\lfloor\log_2(n)\rfloor-1}-1$ is not a counter example (which is obvious since in that case $n$ is odd).
\end{proof}


\begin{theorem}\label{a026644}
    There is exactly one value of $n$ between two consecutive powers of~$2$ such that ${\normalfont\texttt{A268289}}_n=n/2$; this happens with $n < 3\times 2^{\lfloor\log_2(n)\rfloor-1}-1$ and the sequence of such indices is {\normalfont\texttt{A026644}} (except for the initial term of the latter).
\end{theorem}
\begin{proof}
    We prove this by induction in a similar way to the proof of the identity~\ref{encadrement}. We assume that the statement was true for the previous two powers of~$2$; then we use the identity~\ref{oeis6} for finding one new relevant term in the first half of the following interval; we know from the identity~\ref{lemma_half} that no other case occurs in the second half of the interval. The initial relevant indices are $2,4,10,\dots$

    The identity~\ref{oeis6} also gives the rule for building the whole sequence of indices, $a_j = a_{j-1} + 2 a_{j-2}+2$, which is also the building rule for \texttt{A026644} with the same initial terms~$2,4,10,\dots$
\end{proof}

\begin{theorem}
    The minimum value for {\normalfont\texttt{A268289}} between two consecutive powers of~$2$ indices is never reached with an index greater than the relevant term of the sequence~{\normalfont\texttt{A026644}} and this minimum value is at most the corresponding term in the sequence ${\normalfont\texttt{A000975}}$ ---~the ``Lichtenberg sequence''\hspace{.1em}\footnote{The two sequences \texttt{A026644} and \texttt{A000975} are deeply related to the Chinese Rings puzzle.}.
\end{theorem}
\begin{proof}
    According to the identity~\ref{a026644}, ${\normalfont\texttt{A268289}}_n=n/2$ for a given term $n$ of~${\normalfont\texttt{A026644}}$; furthermore, according to the identity~\ref{encadrement}, ${\normalfont\texttt{A268289}}_m\geqslant m/2$; thus greater indices than $n$ will all have greater value han $n/2$. Though the term $n/2$ is not proved to be the minimum one, the sequence will never reach such a low value again. It is known that ${\normalfont\texttt{A000975}}_n={\normalfont\texttt{A026644}}_n/2$.
\end{proof}



\vspace{32pt}
\noindent{\small \textbf{Conflict of Interest:} The authors declare that they have no conflict of interest.}

\vspace{8pt}
\noindent{\small The current article is accessible on \texttt{http://export.arxiv.org/pdf/1908.02250}\hspace{2pt}.}





\begin{thebibliography}{9}

   \bibitem{oeis}
       \emph{The On-Line Encyclopedia of Integer Sequences}, published electronically at \texttt{https://oeis.org}\hspace{2pt}.

   \bibitem{baruchel}
       Thomas \textsc{Baruchel}, \emph{Flattening Karatsuba's recursion tree into a single summation}, in \textit{Computer Science}, ed. Springer Nature, \emph{accepted for publication}. Also available at \texttt{https://arxiv.org/abs/1902.08982} .

   \bibitem{lagarias}
       Jeffrey C. \textsc{Lagarias}, \emph{The Takagi function and its properties}, in Functions in number theory and their probabilistic aspects, 153--189, RIMS Kôkyûroku Bessatsu, B34, Res. Inst. Math. Sci. (RIMS), Kyoto, 2012. MR3014845.   

   \bibitem{takagi}
       Teiji \textsc{Takagi}, \emph{A simple example of the continuous function without derivative}, Tokyo Sugaku-Butsurigakkwai Hokoku, vol.~1, 176--177, 1901.

\end{thebibliography}
\end{document}